\documentclass[10pt,oneside]{amsart}
    
\usepackage{amsfonts}
\usepackage{amsthm}
\usepackage{amssymb}
\usepackage{amsmath}
\usepackage{multicol}
\usepackage{enumitem}
\usepackage{tikz}
\usepackage{mathtools}
\usepackage{graphicx}
\graphicspath{ {documents/} }
\usetikzlibrary{matrix}
\newcommand\mydots{\hbox to 1em{.\hss.\hss.}}
\usepackage[all]{xy}
\theoremstyle{plain}
      \newtheorem{theorem}{Theorem}[section]
      \newtheorem{lemma}[theorem]{Lemma}
      
      \theoremstyle{definition}
      \newtheorem{definition}[theorem]{Definition}
      \theoremstyle{remark}
      
      \theoremstyle{proposition}
      \newtheorem{proposition}[theorem]{Proposition}

      \newcommand{\C}{{\mathbb{C}}}
      \newcommand\blfootnote[1]{%
  \begingroup
  \renewcommand\thefootnote{}\footnote{#1}%
  \addtocounter{footnote}{-1}%
  \endgroup
}

\newcommand{\overbar}[1]{\mkern 1.2mu\overline{\mkern-1.5mu#1\mkern-1.5mu}\mkern 1.2mu}

   		\newcommand{\opn}{\operatorname}
   		
      \makeatletter
      \def\@setcopyright{}
      \def\serieslogo@{}
      \makeatother
      \title{On the basic structure theorems of global singularity theory}
      \author{Natalia Kolokolnikova}
       \address{University of Geneva, Switzerland}
   \email{Natalia.Kolokolnikova@unige.ch}
\begin{document}
\maketitle
\section{Introduction}
\blfootnote{The research was supported by Grant 156645 of the Swiss Science Foundation} Global singularity theory originates from problems in obstruction theory. Consider the following question: is there an immersion in a given homotopy class of maps between two smooth manifolds? We can reformulate this problem as follows. Suppose $M$ and $N$ are smooth real manifolds with $\opn{dim}N \geq \opn{dim} M,$ and $f\colon M \rightarrow N$ is a sufficiently generic smooth map in a fixed homotopy class. The map $f$ is an immersion, if 
$$\Sigma^1 (f)\stackrel{\tiny{def}}{=}\{p \in M\ |\ \opn{dim} \opn{Ker}(d_pf) \geq 1 \}= \varnothing.$$\par
The set $\Sigma^1 (f)$ is called the $\Sigma^1$-\textit{singularity locus}, or simply the $\Sigma^1$-\textit{locus} of $f,$ i.e. the points in $M$ where $f$ has a $\Sigma^1$-\textit{singularity}: the kernel of the differential of $f$ is non-zero. In the case of $\mathbb{Z}_2$-cohomology and a sufficiently generic map $f$, the set $\Sigma^1 (f)$ represents a cohomology class via Poincar\'e duality. Clearly, if the  Poincar\'e dual $\opn{PD}[\Sigma^1 (f)]$ is non-zero in $H^*(M,\mathbb{Z}_2)$, then $f$ is not an immersion. \par
In the 50s, Ren\'e Thom proved the following statement, now known as Thom's principle.
\begin{theorem}[Thom's principle, \cite{Th}]
Let $\Theta$ be a singularity and let $m \leq n$ be non-negative integers. 
Suppose $\{a_1,\mydots,a_m\}$ and $\{ a'_1,\mydots, a'_n\}$ are two sets of graded variables with $\opn{deg}a_i=\opn{deg} a'_i=i.$ Then there exists a universal polynomial in $a_1,\mydots,a_m$ and  $a'_1,\mydots, a'_n$ $$\opn{Tp}[\Theta](a_1,\mydots,a_m,a'_1,\mydots, a'_n)$$ depending only on $\Theta,$ $m$ and $n$, such that for all smooth compact real manifolds $M$ and $N$, $\opn{dim}(M) =m,$ $\opn{dim}(N)=n,$ and a sufficiently generic smooth map $f\colon M\rightarrow N,$   
$$ \overline{\Theta (f)} = \overline{\{ p\in M\ |\ f\ \text{has a singularity of type}\ \Theta\ \text{at}\ p\}}$$ is a cycle in $M$, and $$\opn{PD}[ \overline{\Theta (f)}] =\opn{Tp}[\Theta](w_1 (TM),\mydots, w_m (TM),f^* w_1(TN), \mydots, f^* w_n (TN)) \in H^* (M,\mathbb{Z}_2),$$
where $w_i (TM)$ and $w_j (TN)$ are the Stiefel-Whitney classes of the corresponding tangent bundles.
\end{theorem}\par
This universal polynomial is called the \textit{Thom polynomial} of $\Theta$. In this paper we will give a rigorous construction of this polynomial in the case of complex manifolds.\par 
Thom's principle may also be translated from the real to the complex case. \par 
\begin{theorem}[Thom's principle in the complex case]
	Let $\Theta$ be a singularity and let $m \leq n$ be non-negative integers. 
Suppose $\{a_1,\mydots,a_m\}$ and $\{ a'_1,\mydots, a'_n\}$ are two sets of graded variables with $\opn{deg}a_i=\opn{deg} a'_i=i.$ Then there exists a universal polynomial in $a_1,\mydots,a_m$ and  $a'_1,\mydots, a'_n$ $$\opn{Tp}[\Theta](a_1,\mydots,a_m,a'_1,\mydots, a'_n)$$ depending only on $\Theta,$ $m$ and $n$, such that for all compact complex manifolds $M$ and $N$, $\opn{dim}(M) =m,$ $\opn{dim}(N)=n,$ and a sufficiently generic holomorphic map $f\colon M\rightarrow N,$   
$$ \overline{\Theta (f)} = \overline{\{ p\in M\ |\ f\ \text{has a singularity of type}\ \Theta\ \text{at}\ p\}}$$ is a cycle in $M$, and $$\opn{PD}[ \overline{\Theta (f)}] =\opn{Tp}[\Theta](c_1 (TM),\mydots, c_m (TM),f^* c_1(TN), \mydots, f^* c_n (TN)) \in H^* (M,\mathbb{R}),$$
where $c_i (TM)$ and $c_j (TN)$ are the Chern classes of the corresponding tangent bundles.
\end{theorem}
In fact, a result of Borel and Haefliger \cite{BH} implies that the real Thom polynomial for $\Theta$ may be obtained by substituting the corresponding Stiefel-Whitney classes for the Chern classes in the corresponding Thom polynomial in the complex case.\par
Calculating Thom polynomials is difficult: some progress has been made in the works of B\'erczi and Szenes \cite{BSz}, Rimanyi \cite{Rim}, Feh\'er and Rim\'anyi \cite{FRim}, B\'erczi, Feh\'er and Rim\'anyi \cite{BFR}, Gaffney \cite{Gaf} and Ronga \cite{Ron}.\par
In this paper, we will focus on the properties of Thom polynomials of a particular class of singularities: \textit{contact singularities}. Let $M$ and $N$ be compact complex manifolds such that $\opn{dim}(M)=m$, $\opn{dim}(N)=n$ and $m\leq n.$ Let $f\colon M \rightarrow N$ be a sufficiently generic (see Section \ref{Dam} for the genericity condition) holomorphic map. Let $z_1,\mydots, z_m$ be the local coordinates on a chart $U_p$ centered at a point $p\in M$ and let $y_1,\mydots, y_n$ be the local coordinates on a chart $V_{f(p)}$ centered at $f(p)\in N.$ Let us denote $(f\circ y_j)(z_1,\mydots,z_m)= \sum \alpha_{(i_1,\mydots,i_m)} z_1^{i_1}\mydots z_m^{i_m}$ by $f_j (z_1,\mydots,z_m).$
 Denote the algebra of power series in $z_1,\mydots,z_m$ without a constant term by $\C_0 [[z_1,\mydots,z_m]].$\par
 To each point $p \in M$, we can associate the algebra $$A_f (p)=\C_0[[z_1,\mydots,z_m]]/I\langle f_1,\mydots,f_n\rangle,$$ where $I\langle f_1,\mydots,f_n\rangle$ is the ideal generated by $f_1,\mydots, f_n.$\par
 Suppose $A$ is an algebra. The $\Theta_A$-locus of a map $f$ is defined as 
 $$\Theta_A (f)=\{ p\in M\ |\ A_f (p) \cong A \}.$$\par
By definition, the Thom polynomial is a polynomial in two sets of variables. Damon proved that for contact singularities, the Thom polynomial may be expressed in a single set of variables. This is an important result, which plays a central role in recent advances in global singularity theory, but the proof of Damon's theorem  is hard to find in the literature. In this paper, we will give a modern proof of Damon's result in the complex case.\par
\begin{theorem}[Damon, \cite{Dam}]
	Let $A$ be a finite-dimensional commutative algebra and let $l$ be a non-negative integer. Suppose $\{b_i\}_{i\geq 0}$ is the set of graded variables with $\opn{deg}(b_i)=i.$ There exists a universal polynomial in $\{b_i\}_{i\geq 0}$ depending only on $l$ and on $A$ 
	$$\widetilde{\opn{Tp}}_A^l(b_1,b_2,\mydots),$$
	such that for all compact complex manifolds $M$ and $N$ with $\opn{dim}(N)-\opn{dim}(M)=l$ and for every sufficiently generic holomorphic map $f\colon M\rightarrow N$ 	
	$$\opn{Tp} [\Theta_A ](c_1 (TM), c_2 (TM),\mydots, f^* c_1 (TN), f^* c_2 (TN),\mydots)=\widetilde{\opn{Tp}}_A^l(c_1(f),c_2(f),\mydots),$$
	where the variables $c_i (f)$ are given by 
	$$1+c_1 (f) t + c_2 (f) t^2+\mydots ={{\sum f^* c_i (TN) t^i}\over{\sum c_j (TM) t^j}}.$$
\end{theorem}\par
The universal polynomial $\widetilde{\opn{Tp}}_A^l(c_1(f),c_2(f),\mydots) \in \mathbb{R}[c_1(f),c_2(f),\mydots]$ has an important structural property: it is positive when expressed in the geometrically natural Schur basis of $\mathbb{R}[c_1(f),c_2(f),\mydots]$ (we will discuss the geometric significance of this polynomial basis in Section \ref{Pos}). \par
The Schur polynomials form an important basis of the ring of symmetric polynomials parametrized by partitions. Given an integer partition $\lambda=(\lambda_1,\mydots,\lambda_n)$ such that $\lambda_1\geq \lambda_2\geq \mydots \geq \lambda_n>0$ define the conjugate partition $\lambda^*=(\lambda_1^*,\mydots,\lambda_m^*)$ by taking $\lambda_i^*$ to be the $\text{largest}\ j\ \text{such that}\ \lambda_j\geq i.$ Let us denote by $s_{\lambda} (b_1,\mydots, b_n)$ the expression of Schur polynomials in elementary symmetric polynomials:
$$s_{\lambda}(b_1,\mydots,b_n)=\opn{det} \{b_{\lambda_i^* +j -i}\}_{i,j=1}^{n}.$$\par
The second result for which we give a new proof is the following theorem:
\begin{theorem}[Pragacz, Weber, \cite{Prag}]
	Let $A$ be a finite-dimensional commutative algebra and let $m,\ l$ be non-negative integers. Let $M$ and $N$ be compact complex manifolds such that $\opn{dim}(N)-\opn{dim}(M)=l$ and let $f\colon M\rightarrow N$ be a sufficiently generic holomorphic map. The Thom polynomial of $\Theta_A (f)$ expressed in $c_i (f)$ has positive coefficients in the Schur basis: $$\widetilde{\opn{Tp}}_A^l(c_1(f),c_2(f),\mydots)=\sum \alpha_{\lambda} s_{\lambda}(c_1 (f),c_2(f),\mydots)$$
	where $\alpha_{\lambda}\geq 0.$
\end{theorem}
\section{Preliminaries}
\subsection{Singularity theory}
Let $z_1,\mydots, z_n$ be the standard coordinates on $\C^n.$ Denote by $J^n$ the algebra of formal power series in $z_1,\mydots, z_n$ without a constant term, i.e.
$$J^n=\{h\in \C[[z_1,\mydots,z_n]]\ |\ h(0)=0\}.$$\par
The space of $d$-\textit{jets} of holomorphic functions on $\C^n$ near the origin is the quotient of $J^n$ by the ideal of series with the lowest order term of degree at least $d+1$, i.e. the ideal generated by monomials $z_1^{i_1}\mydots z_n^{i_n}$ such that $\sum i_j =d+1.$ We will denote this ideal by $I\langle \overbar{z}^{d+1} \rangle:$
$$J_d^n=J^n/I\langle \overbar{z}^{d+1} \rangle.$$	\par
As a linear space, the algebra $J_d^n$ may be identified with the space of polynomials in $z_1,\mydots,z_n$ of degree at most $d$ without a constant term.
The space of $d$-jets of holomorphic maps from $(\C^n,0)$ to $(\C^k,0)$, or the space of \textit{map-jets}, is denoted by $J_d^{n,k}$ and is naturally isomorphic to $J_d^n\otimes \C^k.$ In this paper we will assume $n\leq k.$ \par
Composition of map-jets together with cancellation of terms of degree greater than $d$ gives a well-defined map
$$J_d^{n,k}\times J_d^{k,m}\longrightarrow J_d^{n,m}$$
$$(\Psi , \Phi )\mapsto  \Psi \circ \Phi.$$\par
Consider a sequence of natural maps $$J_d^{n,k}\rightarrow J_{d-1}^{n,k}\rightarrow \mydots \rightarrow J_1^{n,k}\cong \opn{Hom} (\mathbb{C}^n,\mathbb{C}^k).$$ 
For $\Psi \in J_d^{n,k},$ the \textit{linear part} of $\Psi$ is defined as the image of $\Psi$ in $J_1^{n,k}:$  $\opn{Lin}\Psi=\opn{Im}\Psi \in J_1^{n,k}$.\par
 Consider the set 
$$\opn{Diff}_d^n=\{\Delta \in J_d^{n,n}\ |\ \opn{Lin}\Delta \ \text{invertible}\}.$$ Previously defined operation $"\circ"$ gives this set an algebraic group structure.\par
Let $\Delta_n \in \opn{Diff}_d^n,\ \Delta_k \in \opn{Diff}_d^k,$ and $\Psi \in J_d^{n,k}.$ The \textit{left-right} action of $\opn{Diff}_d^n\times \opn{Diff}_d^k$ on $J_d^{n,k}$ is given by
$$(\Delta_n,\Delta_k)\Psi=\Delta_n \circ \Psi \circ \Delta_k^{-1}.$$\par
\begin{definition}
Left-right invariant algebraic subsets of $J_d^{n,k}$ are called \textit{singularities}.	
\end{definition}
To a given element $\Psi\in J_d^{n,k}\cong J_d^n\otimes \C^k$, presented as $(\Psi_1,\mydots,\Psi_k)$, $\Psi_i\in J_d^n,$ we can associate an algebra $A_{\Psi}=J_d^n/I\langle P_1,\mydots,P_k\rangle.$ This algebra is \textit{nilpotent}: there exists a natural number $m$ such that $A_{\Psi}^m=0,$ in other words, a product of any $m$ elements of $A_{\Psi}$ is equal to $0.$ $A_{\Psi}$ is nilpotent because $J_d^n$ itself is nilpotent: $(J_d^n)^{d+1}=0.$
\begin{definition}
	Suppose $A$ is a nilpotent algebra. The subset 
$$\Theta_A^{n,k}=\{\Psi \in J_d^{n,k}\ |\ A_{\Psi}\cong A\}$$
is called a \textit{contact singularity}.
\end{definition}
\begin{proposition}[\cite{Arn}]
Let $A$ be a nilpotent algebra: $A^{d+1}=0.$ For $n\geq \opn{dim}(A/A^2)$ and $k$ sufficiently large, $\Theta_A^{n,k}$ is a non-empty, left-right invariant, irreducible quasi-projective algebraic subvariety of $J_d^{n,k}.$	
\end{proposition}
\subsection{Equivariant Poincar\'e dual}
Suppose a topological group $G$ acts continuously on a topological space $M,$ and $Y$ is a closed $G$-invariant subvariety in $M.$ In this section we will define an analog of a Poincar\'e dual of $Y$, which reflects the $G$-action: the equivariant Poincar\'e dual of $Y$.\par
Let $G$ be a topological group and let $\pi \colon EG \rightarrow BG$ be the \textit{universal $G$-bundle}, i.e. a principle $G$-bundle such that if $p\colon E \rightarrow B$ is any principle $G$-bundle, then there is a map $\zeta \colon B \rightarrow BG$ unique up to homotopy and $E\cong \zeta^* EG.$ The universal $G$-bundle exists, is unique up to homotopy and can be constructed as a principle $G$-bundle with contractible total space.\par 
Now we can construct the space with a free $G$-action and the same homotopy type as a fixed before topological space $M$, the \textit{Borel construction}:
\begin{definition}
	The \textit{Borel construction} (also homotopy quotient or homotopy orbit space) for a topological group $G$ acting on a topological space $M$ is the space $EG\times_G M,$ i.e. the factor of $EG\times M$ by the diagonal action: $(gx,gy)\sim (x,y),$ where $g\in G, x\in EG, y \in M.$
\end{definition}
\begin{definition}
	The \textit{equivariant cohomology} of $M$ is the ordinary cohomology for the Borel construction:
	$$H_G^*(M)=H^*(EG\times_G M).$$
\end{definition}\par
Note that since $(EG\times pt )/G=EG/G=BG,$ the equivariant cohomology of a point is  $H_G^*(pt)=H^*(BG).$\par
We would like to define an analog of a Poincar\'e dual in the equivariant case, i.e. when a group $G$ acts on a space $M$ and $Y\subset M$ is a closed $G$-invariant subvariety. We constructed a 'substitute' for the orbit space of $G$-action on $M$: the Borel construction $EG\times_G M.$ Now, $EG\times_G Y$ is again a $G$-invariant subvariety of $EG\times_G M,$ and we want to define a dual of $EG\times_G Y$ in $H^*(EG\times_G M)\cong H_G^*(M). $ However, first we have to deal with the fact that $EG$ is usually infinite-dimensional by introducing an approximation.\par 
\begin{lemma}[\cite{And}]
	Suppose $E_m$ -- any connected space with a free $G$-action, such that  $H^i E_m=0$ for $0<i<k(m),$ where $k(m)$ is some integer. Then for any $M,$ there are natural isomorphisms $$H^i(E_m\times_G M)\cong H^i(EG\times_G M)$$
	for $i<k(m).$
\end{lemma}
Let us fix $EG,$ $BG$ and the approximations $$EG_1\subset EG_2\subset \mydots \subset EG$$ together with $BG_i=EG_i/G.$ We can now consider $EG_m\times_G Y\subset EG_m\times_G M$ -- two finite dimensional spaces. Let $D$ be the codimension of $EG_m\times_G Y$ in $EG_m\times_G M.$\par
Every irreducible closed subvariety of a non-singular variety has a well-defined Borel-Moore homology class \cite{Ful}, so we can define the \textit{equivariant Poincar\'e dual} of $Y$ as follows:
$$\opn{eP}(Y)=[EG_m\times_G Y]_{BM}\in H^{2D}(EG_m\times_G M)=H_G^{2D}(M)$$
for $m$ large enough.\par 
\subsection{The Thom polynomial}
We want to study the equivariant Poincar\'e dual of a closure of a singularity $\overbar{\Theta}\subset J_d^{n,k}.$ It is convenient to restrict the action of $\opn{Diff}_d^n\times \opn{Diff}_d^k$ on $J_d^{n,k}$ to the action of $G=\opn{Gl}_n\times \opn{Gl}_k.$ \par
First, we need to fix $EG,$ $BG$ and the corresponding approximations. Recall that $\C^{\infty}$ is defined as $\{(z_1,z_2,\mydots)\ |\ z_i \in \C,\ \text{only finite number of}\ z_i\ \text{is non-zero} \} $. Fix $E \opn{Gl}_n=\opn{Fr}(n,\infty),$ the manifold of $n$-frames of orthonormal vectors in $\C^{\infty},$ and $B\opn{Gl}_n=\opn{Gr}(n,\infty),$ the Grassmannian of $n$-planes in $\C^{\infty}.$ So, in our case $EG=\opn{Fr}(n,\infty)\times \opn{Fr}(k,\infty)$ and $BG=\opn{Gr}(n,\infty)\times \opn{Gr}(k,\infty).$ The approximations are given by $EG_i=\opn{Fr}(n,i)\times \opn{Fr}(k,i)$ and $BG_i=\opn{Gr}(n,i)\times \opn{Gr}(k,i).$\par
By definition, $$\opn{eP}(\overbar{\Theta})\in H_G^{*}(J_d^{n,k})=H^{*}(BG)=H^*(\opn{Gr}(n,\infty)\times \opn{Gr}(k,\infty)),$$
since $J_d^{n,k}$ is contractible.\par
Let $L_n$ denote the tautological vector bundle over $\opn{Gr}(n,\infty),$ i.e. $$\opn{Gr}(n,\infty)\times \C^{\infty}\supset \{(V,q)\ |\ q\in V\}.$$ Then we can identify $H^*(\opn{Gr}(n,\infty),\C)$ with $\mathbb{C}[c_1,\mydots,c_n],$ where $c_i$ are the Chern classes of $L_n^*$ -- the dual tautological bundle. This observation allows us to define the \textit{Thom polynomial} as follows:
\begin{definition}\label{ThPolDef}
	Let $d,n,k\in \mathbb{N}$ and let $n\leq k.$ Let $\Theta \subset J_d^{n,k}$ be a singularity. The \textit{Thom polynomial} of $\Theta$ is defined as
	$$\opn{Tp}[\Theta](c,c')=\opn{eP}(\overbar{\Theta})\in H^*(\opn{Gr}(n,\infty)\times \opn{Gr}(k,\infty))\cong \mathbb{C}[c_1,\mydots,c_k]\otimes \mathbb{C}[c'_1,\mydots,c'_k],$$
	where $c_i$ are the Chern classes of $L_n^*$ and $c'_i$ -- the Chern classes of $L_k^*.$
\end{definition}\par
The notation $\opn{Tp}[\Theta](c,c')$ comes from the total Chern class: $c=\sum c_i.$\par
The Thom polynomial defined above coincides with the universal polynomial from the Thom's principle due to the following theorem.
\begin{theorem}[\cite{Kaz}]
Let $M$ and $N$ be compact complex manifolds, $\opn{dim}(M)\leq \opn{dim}(N),$ and let $f\colon M\rightarrow N$ be a holomorphic map. If the $\Theta$-locus of $f$ has expected dimension and is reduced, then its fundamental class in $M$ represents the Thom polynomial $\opn{Tp}[\Theta](c,c')$ (defined in \ref{ThPolDef}) evaluated in $c_i (TM)$ and $f^* c_j (TN):$
$$\opn{Tp}[\Theta](c_1(TM),c_2 (TM,\mydots, f^* c_1(TN), f^* c_2 (TN),\mydots).$$
\end{theorem}\par
In this paper we will think of Thom polynomial as defined in \ref{ThPolDef}. For a detailed discussion of the relation between this definition and the Thom's principle, see \cite{BSz} and \cite{Kaz}.

\section{Damon's theorem}\label{Dam}
Before stating and proving Damon's theorem, let us first discuss the relation between Thom polynomials for different singularities.\par
Suppose $A$ is a nilpotent algebra. Fix $k,n,k',n'\in \mathbb{N}$ such that $k\geq n,$ $k'\geq n'$ and $k-n=k'-n'.$ Consider $\Theta_A^{n,k}\subset J_d^{n,k}$ and $\Theta_A^{n',k'}\subset J_d^{n',k'}$ and the corresponding approximations for $N, N' \gg 0$ of the Borel constructions $EG_N\times_G \overbar{\Theta}_A^{n,k} \subset EG_N\times_G J_d^{n,k}$ and ${EG'}_{N'}\times_{G'} \overbar{\Theta}_A^{n',k'} \subset {EG'}_{N'}\times_{G'} J_d^{n',k'}$for $G=\opn{Gl}_n\times \opn{Gl}_k$ and $G'=\opn{Gl}_{n'}\times \opn{Gl}_{k'}.$\par
Suppose $\varphi$ and $h$ in the following diagram are holomorphic.
$$\xymatrix {\Sigma_1=EG_N{\times}_{G}{\overbar{\Theta}}_A^{n,k} \ar @{^{(}->}[r] & EG_N{\times}_{G}\mathcal{J}_d^{n,k}  \ar[r]^-{\pi} \ar[d]^-{h} & \opn{Gr}(n,N)\times \opn{Gr}{(k,N)} \ar[d]^-{\varphi}\\
\Sigma_2=EG'_{N'}{\times}_{G'}{\overbar{\Theta}}_A^{n',k'}  \ar @{^{(}->}[r] & EG'_{N'}{\times}_{G'}\mathcal{J}_d^{n',k'} \ar[r]^-{\pi'} & \opn{Gr}(n',N')\times \opn{Gr}{(k',N')}}$$\par
If the following conditions \cite{FRim} are satisfied:
\begin{itemize}
	\item the square on the right commutes,
	\item $h^{-1}(\Sigma_2)=\Sigma_1$,
	\item $h$ is transversal to the smooth points of $\Sigma_2$,
\end{itemize}
then $h^*\opn{PD}[\Sigma_2]=\opn{PD}[h^{-1}(\Sigma_2)]=\opn{PD}[\Sigma_1]$. From the commutativity of the right square we obtain the equality $$\opn{Tp}[\Theta_A^{n,k}]=\varphi^* \opn{Tp}[\Theta_A^{n',k'}].$$\par
Let now $n'=n+1,\ k'=k+1.$ Define the map $\varphi$ as follows:
$$\varphi \colon \opn{Gr}(n,N)\times \opn{Gr}(k,N)\longrightarrow \opn{Gr}(n+1,N+1)\times \opn{Gr}(k+1,N+1)$$
$$(V_1,V_2)\mapsto (V_1\oplus \C, V_2\oplus C).$$\par
Define $h$ in a similar way. Let $(e_1,e_2,\mydots,e_{N+1})$ be a fixed orthonormal basis of $\C^{N+1},$ and let $(t_1,\mydots,t_n)$ be an orthonormal $n$-frame in $\C^N$ such that $e_{n+1}\notin \langle t_1,\mydots,t_n\rangle.$ Let $(\Psi_1 \mydots,\Psi_k )\in J_d^{n,k},$ i.e. $\Psi_j (z_1,\mydots,z_n)\in J_d^n.$ 
$$h\colon EG_N\times_G J_d^{n,k}\longrightarrow EG'_{N+1}\times_{G'}J_d^{n+1,k+1}$$
$$((t_1,\mydots,t_n),(\Psi_1,\mydots,\Psi_k))\mapsto ((t_1,\mydots,t_n,e_{n+1}),(\Psi_1,\mydots,\Psi_k,z_{n+1}))$$\par
Let us denote the set of Chern classes of the dual tautological bundle $L_n^*$ on $\opn{Gr}(n,N)$ by $c=c_1,\mydots, c_n$, the Chern classes of $L_k^*$ by $c'=c'_1,\mydots,c'_k,$ the Chern classes of $L_{n+1}^*$ on $\opn{Gr}(n+1,N+1)$ by $\overline{c}=\overline{c}_1,\mydots,\overline{c}_{n+1}$ and the Chern classes of $L_{k+1}^*$ by ${\overline{c}}'={\overline{c}'}_1,\mydots,{\overline{c}'}_{k+1}.$ The transversality and the commutativity of the square on the right are straightforward, so the following is true: $$\opn{Tp}[\Theta_A^{n,k}](c,c')=\varphi^* \opn{Tp}[\Theta_A^{n+1,k+1}](\overline{c},\overline{c}').$$\par
We can also show how the pullback of $\varphi$ acts on the Chern classes ${\overline{c}}_i$ and ${\overline{c}'}_i:$
$$\varphi^* ({\overline{c}}_i) = c_i\ \text{for}\ i\leq n\ \text{and}\ \varphi^*({\overline{c}}_{n+1})=0,$$
$$\varphi^* ({\overline{c}'}_i) = c'_i\ \text{for}\ i\leq k\ \text{and}\ \varphi^*({\overline{c}'}_{k+1})=0.$$\par
Using the properties of a pullback map we conclude the following.
\begin{lemma} In the above notations,
$$\opn{Tp}[\Theta_A^{n,k}](c_1,\mydots,c_n,c'_1,\mydots,c'_k)=\opn{Tp}[\Theta_A^{n+1,k+1}](\varphi^*(\overline{c}_1),\mydots,\varphi^* (\overline{c}_{n+1}) ,\varphi^*(\overline{c}'_1),\mydots,\varphi^*(\overline{c}'_k))=$$
$$=\opn{Tp}[\Theta_A^{n+1,k+1}](c_1,\mydots,c_n,0,c'_1,\mydots,c'_k,0)$$
\end{lemma}
We can iterate the same procedure for $\opn{Tp}[\Theta_A^{n+2,k+2}],$ $\opn{Tp}[\Theta_A^{n+3,k+3}],$ etc, but since the Thom polynomial has a fixed degree, there will be a stabilization. This conclusion proves that the Thom polynomial depends only on the difference $k-n$ but not on $n$ and $k$, it also allows us to define the notion that generalizes the Thom polynomial.
\begin{definition}
	Fix a nilpotent algebra $A$ and the difference between the dimensions of the source and the target of the map-jets, i.e. $n-k$ in our previous notations, denote this number by $j.$ Let $m>\opn{codim}(\overline{\Theta}_A^{n,k})$ in $J_d^{n,k}.$
	Define the \textit{universal Thom polynomial} as
	$$\opn{UTp}[\Theta_A^j](c_1,\mydots,c_m,c'_1,\mydots,c_{m+j})=\opn{Tp}[\Theta_A^{n,n+j}](c_1,\mydots,c_m,c'_1,\mydots,c'_{m+j})$$
	for $n>m.$
\end{definition}\par
For all $n,k$ such that $k-n=j$ we obtain
$$\opn{Tp}[\Theta_A^{n,k}](c_1,\mydots,c_n,c'_1,\mydots, c'_k)=\opn{UTp}[\Theta_A^{j}](c_1,\mydots,c_n,0,\mydots,0,c'_1,\mydots,c'_k,0,\mydots,0).$$\par
Let us show an important property of the universal Thom polynomial. Let $$f\colon \opn{Gr}(n,N)\longrightarrow \opn{Gr}(l,M)$$ be any holomorphic map.
Consider the diagram:
$$\xymatrix { EG_N{\times}_{G}\mathcal{J}_d^{n,k}  \ar[r]^-{\pi} \ar[d]^-{h} & \opn{Gr}(n,N)\times \opn{Gr}{(k,N)} \ar[d]^-{\varphi}\\
EG'_{N+M}{\times}_{G'}\mathcal{J}_d^{n+l,k+l} \ar[r]^-{\pi'} & \opn{Gr}(n+l,N+M)\times \opn{Gr}{(k+l,N+M)}}$$\par
Define $\varphi$ as
$$\varphi (V_1,V_2)=(V_1 \oplus f(V_1),V_2 \oplus f(V_1)),\ V_1\in \opn{Gr}(n,N),\ V_2\in \opn{Gr}(k,N).$$\par
Let $(e_1,\mydots,e_n)$ be the orthonormal basis for $V_1,$ $(e'_1,\mydots,e'_k)$ -- the orthonormal basis for $V_2,$ and $(\overline{e}_1,\mydots,\overline{e}_l)$ -- the orthonormal basis for $f(V_1).$ Let $\Psi=(\Psi_1,\mydots, \Psi_k)\in J_d^{n,k}.$ Define $h$ as follows:
$$h[(e_1,\mydots,e_n,e'_1\mydots,e'_k),\Psi]=$$ $$=[(e_1,\mydots,e_n,\overline{e}_1,\mydots,\overline{e}_l,e'_1,\mydots,e'_k,\overline{e}_1,\mydots,\overline{e}_l),(\Psi_1,\mydots,\Psi_k,z_{k+1},\mydots,z_{k+l})]$$\par
Let $c$ be the total Chern class of $L_n^*,$ $c'$ -- the total Chern class of $L_k^*,$ and $d_f$ -- the total Chern class of $f^*(L_{l}^*).$ We have the following formulae for the pullbacks:
$$\varphi^* c(L_{n+l}^*)=c(L_n^*\oplus f^* L_l^*)=cd_f$$
$$\varphi^* c(L_{k+l}^*)=c(L_k^*\oplus f^* L_l^*)=c'd_f$$\par
On the level of the universal Thom polynomials we obtain the following.
\begin{lemma}
In the above notations,
$$\opn{UTp}[\Theta_A^j](c,c')=\opn{UTp}[\Theta_A^j] (cd_f,c'd_f).$$
\end{lemma}
Let us formulate the Damon's theorem.
\begin{theorem}[Damon, \cite{Dam}]
	Let $d,n,k\in \mathbb{N}$ and let $n\leq k.$ Suppose $A$ is a nilpotent algebra and $\Theta_A^{n,k}\subset J_d^{n,k}$ a contact singularity. The Thom polynomial of $\Theta_A^{n,k}$ depends only on the difference $k-n$ and can be expressed in a single set of variables $\tilde{c}$ given by the generating series $$1+\tilde{c}_1 t+\tilde{c}_2 t^2+\mydots = {{\sum_{i=0}^k c'_i t^i}\over{\sum_{j=0}^n c_j t^j}}.$$
\end{theorem}
These new variables are called the \textit{relative Chern classes}. We will denote the Thom polynomial expressed in the relative Chern classes by $\opn{Tp}[\Theta_A^{n,k}](c'/c).$\par
\begin{proof}
The previous discussion implies that if there existed a map $f$ such that $d_f=1/c,$ the Damon's theorem would be proved since
$$\opn{UTp}[\Theta_A^j](c,c')=\opn{UTp}[\Theta_A^j](1,c'/c)=\opn{UTp}[\Theta_A^j](c'/c).$$\par
In fact, such a map does not exist. The equality $c(L^*)=1/c(Q^*)$ holds for a finite Grassmannian, so $d_f$ should be $c(Q^*),$ but the Chern classes of the dual tautological bundle can not be pulled back to $Q^*$ via a holomorphic map because $c(L^*)$ is positive (i.e. the Chern classes of $L^*$ are linear combinations with non-negative coefficients of the Poincar\'e duals to analytic subvarieties) and $c(Q^*)$ is not.\par
Let $S$ be an ample line bundle over $\opn{Gr}(n,N).$ Then for $\alpha$ big enough, $Q_n^*\otimes S^{\otimes \alpha}$ is generated by its global holomorphic section and thus has positive Chern classes. There exists a holomorphic map 
$$f_{\alpha}\colon \opn{Gr}(n,N)\longrightarrow \opn{Gr}(n+l_{\alpha},N+M_{\alpha})$$
such that $f_{\alpha}^* (L_{n+l_{\alpha}}^*)=Q_n^*\otimes S^{\otimes \alpha}.$\par
Let us compute the total Chern class of this twisted bundle. Denote the bundles from the splitting principle for $Q_n^*$ by $E_1,\dots,E_n$ and their first Chern classes by $y_1,\dots,y_n,$ denote the first Chern class of $S$ by $z.$ Then the following identity holds:
$$c(Q_n^*\otimes S^{\otimes \alpha})=c(E_1\otimes S^{\otimes \alpha} \oplus \mydots \oplus E_n \otimes S^{\otimes \alpha})=$$
$$=\prod_{i=1}^n (y_i+\alpha z+1)=\prod_{i=1}^n (y_i+1)+\alpha P(\alpha)=c(Q_n^*)+\alpha \cdot P(\alpha),$$
where $\alpha \cdot P(\alpha)$ is a polynomial in $\alpha$ that contains all the summands of $\prod_{i=1}^n (x_i+\alpha y+1)$ that depend on $\alpha.$
Define $$\varphi_{\alpha}\colon \opn{Gr}(n,N)\times \opn{Gr}(k,N)\longrightarrow \opn{Gr}(n+l_{\alpha}, N+M_{\alpha})\times \opn{Gr}(k+l_{\alpha},N+M_{\alpha})$$
$$(V_1,V_2)\mapsto (V_1\oplus f_{\alpha}(V_1),V_2 \oplus f_{\alpha}(V_1))$$
Denote the total Chern class of the dual tautological bundle $L_{n+l_{\alpha}}^*$ on $\opn{Gr}(n+l_{\alpha}, N+M_{\alpha})$ by $\overline{c}$ and the total Chern class of the dual tautological bundle $L_{k+l_{\alpha}}^*$ on $\opn{Gr}(k+l_{\alpha}, N+M_{\alpha})$ by $\overline{c}'.$ Then by the previous discussion we have the following relations between the Chern classes:
$$\varphi^*(\overline{c})=c \cdot (c(Q_n^*) +\alpha P(\alpha))=1+ c\cdot \alpha P(\alpha)$$
$$\varphi^*(\overline{c}')=c' \cdot (c(Q_n^*)+\alpha P(\alpha))=c'/c + c' \cdot \alpha P(\alpha).$$\par
Or, on the level of the universal Thom polynomials:
$$\opn{UTp}[\Theta_A^j](1+c\cdot \alpha P(\alpha), c'/c + c' \cdot \alpha P(\alpha))=\opn{UTp}[\Theta_A^j](1, c'/c)+\alpha P_2(\alpha)=\opn{UTp}[\Theta_A^j](c,c'),$$
where $\alpha P_2 (\alpha)$ denotes all the summands that depend on $\alpha.$\par
Since $\alpha P_2 (\alpha)= \opn{UTp}[\Theta_A^j](c,c')-\opn{UTp}[\Theta_A^j] (1,c'/c)$ their expressions in the Schur polynomial basis are also equal:
$$\alpha P_2(\alpha)= \alpha \sum W_{\lambda \mu}(\alpha) s_{\lambda}(c)s_{\mu} (c')$$
$$\opn{UTp}[\Theta_A^j](c,c')-\opn{UTp}[\Theta_A^j] (1,c'/c) = \sum B_{\lambda \mu} s_{\lambda}(c) s_{\mu} (c')$$
$$\alpha \sum W_{\lambda \mu}(\alpha) s_{\lambda}(c)s_{\mu} (c') =\sum B_{\lambda \mu} s_{\lambda}(c) s_{\mu} (c') $$
This equation holds if and only if 
$$B_{\lambda \mu}=\alpha W_{\lambda \mu} (\alpha)$$
for all $\lambda$ and $\mu.$ However, since this is true for all sufficiently big $\alpha,$ the polynomial $B_{\lambda \mu}-\alpha W_{\lambda \mu} (\alpha)$ has infinite number of roots. Thus, it is zero for all $\alpha.$ This implies that $B_{\lambda \mu}=0$ for all $\lambda$ and $\mu,$ i.e. $\opn{UTp}[\Theta_A^j](c,c')=\opn{UTp}[\Theta_A^j] (1,c'/c).$
\end{proof}
\section{Positivity}\label{Pos}
The \textit{Schur polynomials} serve as a natural basis for the cohomology ring of  Grassmannians. Given an integer partition $\lambda=(\lambda_1,\mydots,\lambda_n),$ such that $N\geq \lambda_1\geq \lambda_2\geq \mydots \geq \lambda_n>0$ define the conjugate partition $\lambda^*=(\lambda_1^*,\mydots,\lambda_m^*)$ by taking $\lambda_i^*$ to be the $\text{largest}\ j\ \text{such that}\ \lambda_j\geq i.$ Denote by $s_{\lambda} (b_1,\mydots, b_n)$ the expression of the Schur polynomials in elementary symmetric polynomials:
$$s_{\lambda}(b_1,\mydots,b_n)=\opn{det} \{b_{\lambda_i^* +j -i}\}_{i,j=1}^{n}.$$\par
The Schur polynomials of degree $d$ in $n$ variables form a linear basis for the space of homogeneous degree $d$ symmetric polynomials in $n$ variables.\par
Consider the finite Grassmannian $\opn{Gr}(n,N).$ The Schur polynomials indexed by $\lambda$ such that $N\geq \lambda_1\geq \mydots \geq \lambda_n >0,$ evaluated in the Chern classes $c_1,\mydots, c_n$ of the dual tautological vector bundle $L_n^*$ are the Poincar\'e duals of the Schubert cycles -- homological classes of Schubert varieties $\sigma_{\lambda}$, special varieties forming a basis for the homology of the Grassmannian \cite{Ful}:
$$s_{\lambda}(c_1,\mydots,c_n)=\opn{PD}[\sigma_{\lambda}].$$\par
The following result was first proved by Pragacz and Weber. Here we give a new proof of this result.
\begin{theorem}[Pragacz, Weber, \cite{Prag}]
	Let $d,n,k\in \mathbb{N}$ and let $n\leq k.$ Suppose $A$ is a nilpotent algebra and $\Theta_A^{n,k}\subset J_d^{n,k}$ a contact singularity. The Thom polynomial of $\Theta_A^{n,k}$ expressed in the relative Chern classes is Schur-positive:
	$$\opn{Tp}[\Theta_A^{n,k}](c'/c)=\sum \alpha_{\lambda} s_{\lambda}(c'/c)$$
	where $\alpha_{\lambda}\geq 0.$
\end{theorem}
\begin{proof}
By Damon's theorem, Thom polynomials for contact singularities can be written as follows:
$$\opn{Tp}[\Theta_A^{n,k}](c,c')=\opn{Tp}[\Theta_A^{n+M,k+M}](1,c'/c)=$$ $$=\sum_{\lambda}\alpha_{0 \lambda }s_{0}(1)s_{\lambda}(c'/c)=\sum_{\lambda}\alpha_{0 \lambda }s_{\lambda}(c'/c)$$ 
for $M$ big enough. To prove the positivity we show that $\alpha_{0 \lambda }\geq 0$ for all $\lambda.$\par
Fix a plane $V_0\in \opn{Gr}(n,N)$ and define the map $$h\colon \opn{Gr}(k,N)\longrightarrow \opn{Gr}(n,N)\times \opn{Gr}(k,N)$$  $$h(V)=(V_0,V).$$\par 
Let $\varphi$ be the unique map making the following diagram commutative: 
$$\xymatrix {\varphi^{-1}(EG_N\times_G \overbar{\Theta}_A^{n,k})\ar @{^{(}->}[r] & h^*(EG_N\times_G J_d^{n,k}) \ar[r]^-{p_2} \ar[d]^-{\varphi} & \opn{Gr}(k,N) \ar[d]^-{h}\\
\Sigma=EG_N\times_G \overbar{\Theta}_A^{n,k} \ar @{^{(}->}[r] & EG_N\times_G J_d^{n,k} \ar[r]^-{p_1} & \opn{Gr}(n,N)\times \opn{Gr}(k,N)}$$\par
The idea of the proof is to show that $$\sum_{\lambda}\alpha_{0 \lambda }s_{\lambda}(c') = h^* (\opn{Tp}[\Theta_A^{n,k}](c,c'))=\opn{PD}[X],$$ where $X$ is an analytic cycle in $\opn{Gr}(k,N).$\par
Let $\sigma_{\lambda'}$ be a homology class of a Schubert variety of dimension complementary to $\opn{dim}X$. $\opn{Gl}(k)$ acts transitively on $\opn{Gr}(k,N),$ so by Kleiman's theorem \cite{Kl} there exists $A \in \opn{Gl}_k$ such that $(A X) \cap \sigma_{\lambda '}$ is of expected dimension (so, discrete) and $A X$ is homologous to $X.$ 
$$\# (X\cap \sigma_{\lambda '})=\opn{PD}[X]\cdot\opn{PD}[\sigma_{\lambda'}]=\sum_{\mu}\alpha_{0 \mu }s_{\mu}(c')s_{\lambda'}(c')= \alpha_{0 \lambda }=$$
$$=\# (AX \cap \sigma_{\lambda'})=\sum_{x\in CX\cap \sigma_{\lambda '}}\opn{mult}_x\geq 0.$$\par
Here $\opn{mult}_x$ is an intersection multiplicity, which is non-negative for two analytic cycles.\par 
Let us consider the details. We should construct the algebraic variety $X.$ First, denote $EG_N\times_G J_d^{n,k}$ by $E$ and $EG_N\times_G \overbar{\Theta}_A^{n,k}$ by $\Sigma$ for short. It is clear that ${\varphi}^{-1}(\Sigma)\subset h^*(E).$ If $\varphi$ is also transversal to $\Sigma,$ then we have that $${\varphi}^* \opn{PD}[\Sigma]=\opn{PD}[{\varphi}^{-1}(\Sigma)].$$\par
By definition, we need to show that: $$\opn{Im}(d_x(\varphi))+T_{\varphi (x)}\Sigma=T_{\varphi (x)} E$$ for $x\in \varphi^{-1}(\Sigma).$ Locally $$T_{(z,y)} E=T_{z}(\opn{Gr}(n,N)\times \opn{Gr}(k,N))\oplus T_{y}\mathcal{J}_d^{n,k}$$ for $z\in EG_N=\opn{Gr}(n,N)\times \opn{Gr}(k,N)$ and $y\in J_d^{n,k}.$ With this interpretation the transversality is obvious since $\opn{Im}(d_x(\varphi))$ has $T_{y}\mathcal{J}_d^{n,k}$ as a direct summand and $T_{\varphi (x)}\Sigma$ has $T_{z}(\opn{Gr}(n,N)\times \opn{Gr}(k,N))$ as a direct summand.\par
Let us show that the vector bundle $h^*(E)$ has enough holomorphic sections to find a holomorphic section $s$ transversal to $\varphi^{-1}(\Sigma).$ 
\begin{lemma}
$EG_N\times_G J_d^{n,k}=\left({\bigoplus}_{i=1}^d\opn{Sym}^i L_n \right)\boxtimes L_k^*$ 
\end{lemma}
\begin{proof}
An element of a fiber of $EG_N\times_G J_d^{n,k}$ is a class $[(e_n,e_k),f],$ where $f \in \mathcal{J}_d^{k,n},$ $e_n$ is a frame, i.e. a linear injective map form $\C^n$ to $\C^N,$ and $e_k$ is a linear injective map form $\C^k$ to $\C^N.$ We consider a class $[(e_n,e_k),f]$ with the equivalence relation $$((e_n,e_k),f) \sim ((e_n A_n^{-1},e_k A_k^{-1}),A_k f A_n^{-1}),$$\par
where $A_n \in \opn{Gl}_n,$ $A_k \in \opn{Gl}_k.$\par
An element of the fiber of $\left({\bigoplus}_{i=1}^d\opn{Sym}^i L_n \right)\boxtimes L_k^*$ is a polynomial function of degree at most $d$ without a constant term between $V_n \in \opn{Gr}(n,N)$ and $V_k \in \opn{Gr}(k,N).$\par
The map $[(e_k,e_n),f] \mapsto e_n \circ f \circ e_k^{-1}$ is correctly defined and is a bijection.
\end{proof}
We use this lemma to 'decompose' $h^*(E):$
$$h^*(E)=\left({\bigoplus}_{i=1}^d \opn{Sym}^i(\opn{Triv}_n)\right)\otimes L_k^*=\opn{Triv}_{{d+n \choose n}-1}\otimes L_k^*,$$
where $\opn{Triv}_n$ is a trivial vector bundle whose fiber is a complex vector space of dimension $n.$\par
We use the following theorem to show that this bundle has enough global holomorphic sections to find one transversal to $\varphi^{-1}(\Sigma).$
\begin{theorem}[Parametric transversality theorem, \cite{Hir}]
	Let $M$, $K$, $Z$, $S$ be smooth manifolds. Consider $F\colon M\times S \rightarrow K\supset Z$, smooth map transversal to $Z.$ Then for almost all $s\in S$ the map $F_s$ is transversal to $Z.$ 
\end{theorem}
Let $D={{d+n \choose n}-1}.$ In the notations of Thom's transversality theorem, let $$M=\opn{Gr}(k,N),\ K=\opn{Hom}(L_k,\C^{D})\cong h^* (EG_N\times_G J_d^{n,k}),$$
$$Z=\varphi^{-1}(\Sigma),\ S=\Gamma(\opn{Hom}(L_k,\C^D))=\opn{Hom}(\C^N,\C^D).$$\par
Then, the map $F$ from the theorem is the following:
$$F\colon \opn{Gr}(k,N)\times \opn{Hom}(\C^N,\C^D)\longrightarrow \opn{Hom}(L_k,\C^D).$$
$$(V,f)\mapsto f|_V.$$\par
The transversality of $F$ to $\varphi^{-1}(\Sigma)$ obviously follows from the fact that $d_{(V,f)}F$ is surjective for all $V$ and $f$.\par
Now, by Parametric transversality theorem, the set of holomorphic sections of $h^*(E)$ transversal to smooth points of ${\varphi}^{-1}(\Sigma)$ is open and dense in all holomorphic sections of this bundle. The set of holomorphic sections of $h^*(E)$ transversal to smooth points of the set of singular points of ${\varphi}^{-1}(\Sigma)$ is open and dense in the set of holomorphic sections transversal to smooth point of ${\varphi}^{-1}(\Sigma)$, and so on. Since this procedure drops the dimension of the variety, it is the finite process and the intersection of a finite number of open and dense sets is again open and dense. So, we can choose a holomorphic section $s$ transversal to ${\varphi}^{-1}(\Sigma).$ \par  
 The analytic subvariety $X$ from the discussion at the beginning of the proof is $s^{-1}{\varphi}^{-1}(\Sigma):$ $$\opn{PD}[s^{-1}{\varphi}^{-1}(\Sigma)]=s^*\opn{PD}[{\varphi}^{-1}(\Sigma)]=({{pr_2}^*})^{-1}{\varphi}^*\opn{PD}[\Sigma]=h^* ({{pr_1}^*})^{-1}\opn{PD}[\Sigma]=$$
 $$=h^* \opn{Tp}[\Theta_A^{n,k}](c,c')=\sum_{\lambda}{\alpha}_{0 \lambda }s_{\lambda}(c')$$ and the proof of positivity is complete.		
\end{proof}

\end{document}